\documentclass{amsart}
\usepackage{amssymb,amsmath,amsrefs}
\usepackage[colorlinks=true]{hyperref}
\hypersetup{urlcolor=blue, citecolor=green}
\usepackage[dvips]{graphicx}
\usepackage{color}
\usepackage{subcaption}
\usepackage{faktor}
\usepackage{float}
\usepackage{cancel}

\graphicspath{{Imagens/}}
\theoremstyle{plain}

\newtheorem{mainthm}{Theorem}

\newtheorem*{conj*}{Conjecture}
\newtheorem*{cor*}{Corollary}
\newtheorem{theorem}{Theorem}[section]

\newtheorem{proposition}[theorem]{Proposition}

\newtheorem{lemma}[theorem]{Lemma}

\theoremstyle{definition}
\newtheorem*{def*}{Definition}
\newtheorem{remark}[theorem]{Remark}
\newtheorem{rmk}[theorem]{Remark}
\newtheorem{example}[theorem]{Example}

\newtheorem{definition}[theorem]{Definition}

\renewcommand{\epsilon}{\varepsilon}

\newcommand{\R}{\mathbb{R}}
\newcommand{\eps}{\varepsilon}

\title{Entropy of Singular Suspensions.}

\author{Elias Rego and Sergio Roma\~{n}a}

 \subjclass[2020]{Primary: 37C10, Secondary: 37B40.}

\keywords{Topological Entropy, Singular Flows, Suspension Flows, Minimal Sets.}

        \thanks{The first was partially supported by  NSFC 12250710130 from China. The second author was supported by  “Bolsa Jovem Cientista do Nosso Estado No. E-26/201.432/2022, Brazil, NNSFC 12071202, and NNSFC 12161141002 from China.}
\date{}

\begin{document}
\maketitle

\begin{abstract}
In this work, we investigate diffeomorphisms whose positiveness of topological entropy is destroyed by singular suspensions. We show that this phenomenon is rare in the set of  $C^1$-diffeomorphisms. Precisely, we prove that for an open and dense set of $C^1$-diffeomorphism positive topological entropy is preserved by singular suspensions, even for suspensions with infinitely many singularities. We prove a similar result to the conservative diffeomorphisms.  We apply our techniques to show that every expansive singular suspension $C^{1}$-flow over a three-dimensional manifold has positive topological entropy. Finally, we explore this phenomenon for Anosov dynamics, showing  that to nullify the topological entropy for Anosov suspension flow, the set of singularities must capture the non-wandering set.
\end{abstract}

\section{Introduction}\label{Intro}

The foundation of modern dynamical systems theory can be traced back to the pioneering works of Henri Poincaré on the qualitative theory of differential equations (\cite{Po}). Over time, this theory has blossomed into a rich field within mathematics. From a modern perspective, a dynamical system is defined as a triple $(G,M,\phi)$, where $G$ is a group,  $M$ is set endowed with some structure and $\phi$ is an action of $G$ on $M$ which agrees with the structure of $M$. When we replace $G$ by $\mathbb{Z}$ or $\mathbb{R}$, we obtain the classical discrete-time and continuous-time systems, respectively. 

If $M$ is a smooth closed manifold and  $\phi$ is a continuous-time system of class $C^1$, we are on the classical setting of flows, which is deeply related to the Poincar\'{e}'s work where it is evidenced a deep connection between the discrete-time and the continuous-time systems: The poincar\'{e} maps. Indeed, the poincar\'{e} maps form a tool to discretize time in flow. This discovery led the community to study the dynamics of diffeomorphisms to obtain results that could be used to understand the dynamical behavior of flows. Nevertheless, discrete-time systems have proven to be a proper theory with several results exclusive to the diffeomorphism scenario.  So a very natural question arose:

\textit{\textbf{Question 1:} Suppose that some result holds for diffeomorphisms. Is it true that the same or an analogous result holds for flows?   }

\vspace{0.1in}
In the converse direction of the Poincar\'{e} maps, there is a natural way of obtaining a flow from a diffeomorphism: The so-called suspension flows. This makes suspension flows the more natural candidates to explore Question 1. In the sequel, we recall its definition. 

Let $M$ be a smooth closed Riemannian manifold. Let $f: M\to M$ be a diffeomorphism. Let $\rho: M\to \mathbb{R}$ be a continuous function (the roof function). Define a new manifold by taking the $\overline{M}$  as the quotient space of $M\times \mathbb{R}$, through the equivalence relation $(x,\rho(x))\sim (f(x),0)$. It is well known that $\overline{M}$ is a closed and smooth Riemannian manifold.

\begin{definition}\label{Def-Susp}
    The suspension flow of $f$ with roof $\rho$ is the flow $\phi: \mathbb{R}\times \overline{M}\to \overline {M}$ defined 
  by setting $\phi(t,(x,s))=(f^n(x),s')$, where $n$ and $s'$ satisfy
$$\sum_{j=1}^{n-1}\rho(f^j(x))+s'=t+s, \,\,\,\, \,\,\, 0\leq s'\leq \rho(f^n(x)).$$
We will refer to the set  $\{(x, t): 0 \leq t < \rho(x)\}$ as the vertical segment of $x$.
\end{definition}

We then have the following natural question:

\vspace{0.1in}
\textit{\textbf{Question 2:} Suppose that some property holds for a diffeomorphism $f$. Is the same or an analogous property true for suspensions of $f$?   }

\vspace{0.1in}

Several authors searched for dynamical properties that have been preserved by the suspension process in the past decades. For instance, it was proved that expansiveness, shadowing, and hyperbolicity are preserved by suspension (see \cite{BW}, \cite{Th} and \cite{Katok}). On the other hand, other properties like the specification property, are only preserved for specific roof functions.  It was proven in \cite{Ohn} that the suspension process does not necessarily preserve the values of the topological entropy, but its positivity is preserved.

Besides, a typical phenomenon exclusive to the flow theory makes their study more delicate, namely,  the existence of singularities. Let us recall this concept. Let $M$ be a compact and boundaryless Riemannian manifold and $\phi$ be a $C^1$-flow on $M$ and let $X_{\phi}$ denote its velocity vector field. A point $p\in M$ is a singularity for $\phi$ if $X(p)=0$. The existence of singularities accumulated by regular orbits is an obstruction for hyperbolicity as evidenced by the paradigmatic Lorenz Geometric Attractor (\cite{ABS} and \cite{Gu1}). 
In \cite{ALRS} it was proved that the shadowing property also cannot live together with hyperbolic singularities on flows with recurrence. 

The suspension process afore-described gives rise to a non-singular flow. Besides, another process provides a singular flow related to the diffeomorphism $f$. Let $f:M\to M$ be  $C^1$-diffeomorphism and let $\phi:\mathbb{R}\times \overline{M}\to \overline{M}$ be its suspension flow. Let $X_{\phi}$ be the velocity vector field of $\phi$. Fix  a compact subset $S\subset \overline{M}$ and let $\alpha:\overline{M}\to \mathbb{R}$ be a non-negative smooth function such that $\alpha(x)=0$ if, and only if, $x\in S$. 

\begin{definition}
    The singular suspension of $f$ with roof function $\rho$ and a brake $\alpha$ is the flow $\psi$ on $\overline{M}$ induced by the vector field $X=\alpha X_{\phi}$. 
\end{definition}    
    \begin{remark}
        
     The choice for naming the function $\alpha$ as a brake is justified by the effect of $\alpha$ on $X_{\phi}$, which changes the original velocity, stops the points on $S$, and slows down every point approaching the set $S$. 

\end{remark}

This concept was previously explored in several works, such as \cite{Ko2} and \cite{SYZ} where assumed that $S$ is a singleton and  \cite{Shi} which allows $S$ to be a general compact set. It is clear that $S$ is the set of singularities of $\psi$. As we will see later, the properties of the set $S$ have a deep influence on the dynamical properties of the singular suspensions. The previous definition naturally leads us to the following modification of Question 2:

\vspace{0.1in}
\textit{\textbf{Question 3:} Suppose that some property holds for a diffeomorphism $f$. Is it true that the same or an analogous result property for singular suspensions of $f$?}

\vspace{0.1in}

It's easy to see that the simple addition of a singularity to the suspension of a hyperbolic diffeomorphism is enough to rule out the hyperbolicity of its suspension. In \cite{Shi}, it is proved the shadowing property is not preserved by singular suspensions, while in \cite{Ko2}, it was proved that if $S$ is a singleton, then $\psi$ has the finite shadowing property if, and only if, $f$ has the finite shadowing property. Also, in \cite{Shi}, we can see that the singular suspension of an expansive homeomorphism with a single singularity is expansive in the sense of \cite{Ko1}. In this work our focus shifts to the topological entropy of singular suspensions. 

In \cite{SYZ} it was proved that the positiveness of topological entropy is not preserved by equivalence for singular smooth flows. They achieved this result by constructing two different singular suspensions $\psi_1$ and $\psi_2$ of a minimal diffeomorphism $f$ with positive topological entropy.  The two flows $\psi_1$ and $\psi_2$ share the same singularity $\sigma$, but their respective brakes $\alpha_1$ and $\alpha_2$ decay to zero at $\sigma$ with different speeds. Consequently, they obtain a flow with positive entropy, while the other one has zero entropy.
These examples make us aware of the importance of the decay rate of the function $\alpha$ to preserve positive entropy.    This highlights a stark contrast between suspensions and singular suspensions regarding topological entropy. The examples illustrated in \cite{SYZ} lead us to the main question addressed in this work:

\vspace{0.1in}
\textit{\textbf{Main Question:} How easy is it to break positive entropy through the singular suspension process?}
\vspace{0.1in}

The first and easiest answer to this question is that it depends on the set $S$. Indeed, we will see in Section \ref{characsection} that breaking positive entropy through singular suspension is pretty easy if one allows $S$ to be uncountable. So, we need to be more restrictive on the set $S$ to investigate the main question.  Our first main result states that if we allow $S$ to be countable, it is very rare to find a diffeomorphism whose positiveness of topological entropy is broken by singular suspensions.
\begin{mainthm}\label{generic}
   Let $M$ be a  closed Riemannian manifold. There is an open and dense subset $\mathcal{U}\subset \emph{Diff}^1(M)$ such that if $f\in \mathcal{U}$ and  $\psi$ is a singular suspension of $f$ with countable many singularities, then either both $f$ and $\psi$ have positive topological entropy or both $f$ and $\psi$ have zero topological entropy.

\end{mainthm}

    At this point, we would like to highlight we do not assume any extra hypothesis on the decay rate of $\alpha$ to obtain Theorem \ref{generic}. Indeed, while we are on the set $\mathcal{U}$ the decay rate of $\alpha$  does not matter for the entropy. This evidences the novelty and strength of our results.  
    Similarly, all the other results in this work are free from constraints on the decay rate of $\alpha$. 

The main ingredients to prove the previous result are Theorems \ref{equiv}, where it is obtained as a criterion for zero entropy of singular suspensions, and Theorem \ref{entrhs}, where we obtain positive entropy trough measures supported on minimal subsystems of $f$. Our methods also allow us to obtain a stronger result in the conservative world. In the next result, $\text{Diff}^1_{m}(M)$ denote the space of conservative $C^1$-diffeomorphimsms of $M$.

\begin{mainthm}\label{generic2}
   Let $M$ be a smooth closed Riemannian manifold. There is an open and dense subset of $\mathcal{U}\subset \emph{Diff}^1_{m}(M)$ such that if $f\in \mathcal{U}$ and  $\psi$ is a singular suspension of $f$ with countable many singularities, then $\psi$ has positive topological entropy.

\end{mainthm}

If we allow stronger regularity on $f$ we can obtain a stronger result in the low-dimensional scenario.

\begin{mainthm}\label{3d}
   Let $\psi$ be a three-dimensional $C^{1+\epsilon}$-flow obtained from the singular suspension of a diffeomorphism $f$. If $Sing(\psi)$  is countable, either $f$ and $\psi$ have positive topological entropy or both $f$ and $\psi$ have zero topological entropy. 
\end{mainthm}

We apply our techniques to explore the entropy theory of expansive flows. Expansiveness is a topological property of dynamical systems deeply related to the concept of hyperbolicity.  As with many other dynamical properties, expansiveness does not behave well together with the existence of singularities. Indeed, the concept of expansiveness for flows was first introduced in \cite{BW} to study the expansive behavior of Axiom A flows. This concept of expansiveness was proven to be inefficient in covering the types of expansiveness displayed by singular flows, leading to the work \cite{Ko1}, where the concept of $k^*$-expansiveness was introduced to cover the behavior of the Geometric Lorenz Attractor. The  $k^*$-expansiveness is weaker than the expansiveness introduced in \cite{BW}. Nevertheless,  these concepts are equivalent in non-singular scenarios, and therefore every time we mention expansiveness we are referring to the $k^*$-expansiveness (see Section \ref{expsection} for the precise definition).

Expansive diffeomorphisms are known to have positive entropy (see \cite{Ka}). In \cite{ACP} it is proven that every non-singular expansive flow has positive topological entropy, then a natural question is whether expansive singular flows have positive topological entropy. The answer to this question in full generality is negative. Indeed, in \cite{Ar3} it is presented as an example of singular expansive flow on a surface, which must have zero topological entropy by \cite{Y}. In this way, we need to remove the two-dimensional scenario to keep this question meaningful. Partial results were obtained in \cite{AAR} for flows with Lyapunov-stable non-singular subsets. In this work, we shall investigate this question for expansive flows obtained as singular suspensions of diffeomorphisms. precisely, we have the following result: \\

\begin{mainthm}\label{expentr}
    Let $\psi$ be an expansive $C^{1}$-flow on a three-dimensional smooth and closed Riemannian manifold. Suppose $\psi$ is obtained as a singular suspension of a diffeomorphism $f$,  then $\psi$ has positive topological entropy.
\end{mainthm}

\begin{rmk}
    We remark that even though Theorems \ref{3d} and \ref{expentr} are stated in the three-dimensional scenario, Theorem \ref{expentr} is not a consequence of Theorem \ref{3d}. Indeed, to obtain Theorem \ref{expentr} we need to apply different techniques in order to handle the $C^1$-regularity.    
    \end{rmk}

\begin{example}
    Let $f:M\to M$ be an expansive $C^{1+\eps}$-diffeomorphism on a surface $M$. Let $\phi$ be the suspension of $f$ with roof function $\rho\equiv 1$. Fix $p\in \overline{M}$ and set $S=\{p\}$. Consider any brake $\alpha$ which is zero only at $S$. Then, by Theorem 3.5 in \cite{Shi}, the singular suspension $\phi$ with brake $\alpha$ is expansive. In particular, it has positive topological entropy due to Theorem \ref{expentr}.
\end{example}
The previous example gives us a family of expansive flows to which Theorem \ref{expentr} applies. As we will see in Section \ref{expsection}, we will expand this class by showing that $S$ does not need to be a singleton. \\
The richest hyperbolic dynamics to which all the above theorems apply are Anosov diffeomorphisms (conservative or not). For this type of dynamics, there are no countable singularities that nullify the topological entropy of its singular suspension flow. Our final result shows that it is very expensive to nullify the topological entropy of a singular suspension flow from an Anosov diffeomorphism. More precisely
\begin{mainthm}\label{Anosov-Susp}
    Suppose $f$ is an Anosov diffeomorphism and $\psi$ is a singular suspension of $f$. Then $h(\psi)=0$ if, and only if, $\Omega(f)\subset\pi(Sing(\psi))$.
\end{mainthm}
In the previous result $\Omega(f)$ denotes the non-wandering set of $f$, while $\pi$ denotes the natural projection from $\overline{M}$ to the fiber $M\times\{0\}$. In other words, Theorem \ref{Anosov-Susp} tells us that the only way  of destroying the topological entropy of an Anosov diffeormorphism is by choosing a function $\alpha$ that destroys any non-trivial form of recurrence in $\psi$. 

In Section \ref{exesection} we will see some examples where our results apply to. Moreover, we will see that since our techniques are of topological character, they can also be applied in the scope of topological dynamics. \\

\noindent \textbf{Structure of the Paper:}
In section \ref{Preliminaries} We give precise definitions for the main concepts used during this work and state some classical results. In section \ref{characsection} we study in more detail the entropy properties of singular suspensions. In Section \ref{secproofabc} we provide the proofs for Theorems \ref{generic}, \ref{generic2}, and \ref{3d}. Section \ref{expsection} is devoted to the study of expansive singular suspensions and to the proof of Theorem \ref{expentr}. Finally, in section \ref{exesection} we prove Theorem \ref{Anosov-Susp}, present some examples, and propose a discussion of the applicability of our techniques. 

\section{Preliminaries}\label{Preliminaries}
In this section, we give the precise definitions for the concepts used in this text. We also provide some results and explain some technologies employed to obtain our results.  Before we go into detail, we would like to emphasize that all the manifolds considered in this work are compact, boundaryless, and Riemannian, and with the metric induced by their  Riemannian metrics. Also, recall that for a subset $\mathcal{R}$ of a  complete metric space $(M,d)$ is residual if it is a countable intersection of open and dense subsets of $M$.

\subsection{Discrete-Time Dynamical Systems}

Let $(M,d)$ be a closed manifold and let $f: M \to M$ be a 
 diffeomorphism. For any $x\in M$ we defined the \textit{orbit of $x$} under $f$ as the set $$\mathcal{O}_f(x)=\{f^n(x), n\in \mathbb{Z}\},  $$
 where $f^0$ denotes the identity map of $M$ and $f^n$ denotes the $n$-fold composition of $f$ and $f^{-1}$ for $n>0$ and $n<0$, respectively. A subset  $K\subset M$ is said to be $f$-invariant if $f(K)=K$. We say that a point $x\in M$ is non-wandering of $f$ if, for every neighborhood $U$ of $x$, there is $n>0$ such that $f^{-n}(U)\cap U\neq\emptyset$. We denote by $\Omega(f)$ the set of non-wandering points of $f$ 

\vspace{0.2in}
\noindent \textit{Topological Dynamics}
 \vspace{0.1in}

 \begin{definition}[Topological Transitivity]
 We say that  $f$ is topologically transitive if there is $x\in M$ whose orbit is dense in $M$.  A point with a dense orbit is called a transitive point for $f$.     
 \end{definition}
 
\begin{remark}
    It is well known that if $f$ is transitive, the set of transitive points is a residual subset of $M$. 
\end{remark}

\begin{definition}[Minimality]
    We say that $f$ is minimal if $M$ does not contain a compact and proper $f$-invariant subset.
 \end{definition}

It is classical that $f$ is minimal if, and only, every point in $M$ is transitive. Next, we recall the concept of shadowing property. A sequence $(x_n)_{n\in \mathbb{Z}}$ is  a $\delta$-pseudo orbit if $d(f(x_n),x_{n+1})\leq \delta$, for every $i\in \mathbb{Z}$. We say that a $\delta$-pseudo-orbit $(x_n)$ is $\epsilon$-shadowed if there is $y\in M$ such that $d(f^n(y),x_n)\leq\eps$, for every $n\in\mathbb{Z}$.

\begin{definition}
     We say that $f$ has the shadowing property if, for every $\eps>0$, there is $\delta>0$ such that every $\delta$-pseudo-orbit is $\eps$-shadowed by some point in $M$. 
\end{definition}

Another property that will be important for our exposition is the sensitivity to initial conditions. 

\begin{definition}[Senstitivity]
    We say that $f$ is sensitive to the initial conditions if there is $c>0$ such that for every $x\in M$ and every neighborhood  $U$ of $x$, there is $y\in U$ and $n>0$ such that $d(f^n(y),f^n(x))>c$. 
\end{definition}

Now we define topological entropy, the main concept investigated in this work. We first give the concept of topological entropy for subsets. Let  $N$ be a compact subset of $M$ (not necessarily invariant). Fix $\varepsilon, n>0$. A subset $K\subset N$ is called a  $n$-$\varepsilon$-separated subset of $N$ if for any pair of distinct points $x,y\in K$ there is  $0\leq n_0\leq n$ such that $$d(f^{n_0}(x),f^{n_0}(y))>\varepsilon.$$ 
Let $S(n,\varepsilon, N)$ denote the maximal cardinality of a $n$-$\varepsilon$-separated subset of $N$.  
 
 \begin{remark} It is well known that $S(n,\varepsilon,N)$  is always finite due the compactness of $M$.
 \end{remark}

\begin{definition}

    We define the topological entropy of $f$ on $N$ as the number  $$h(f,N)=\lim\limits_{\varepsilon\to 0}\limsup\limits_{n\to\infty}\frac{1}{n}\log(S(n,\varepsilon, N)).$$
    The topological entropy of $f$ is defined as the number $$h(f)=h(f,M). $$
\end{definition}

\noindent \textit{Ergodic Theory}
\vspace{0.1in}

 Through this work, we will assume that $M$ is measurable space endowed with its Borel $\sigma$-algebra. Next, recall that a probability measure $\mu$ is an \textit{invariant measure} for the diffemorphism $f$ if $\mu(f^{-1}(A))=\mu(A)$ for every Borelian subset $A$. Denote by $\mathcal{M}(f)$ the set of invariant measures for the flow of $X$.  A $f$-invariant measure is \textit{ergodic} if $\mu(A)\mu(A^c)=0$, for every borelian and $f$-invariant set $A$. Let $\mathcal{M}_e(f)$  denote the set of ergodic measures for $f$. The \textit{metric entropy} of an $f$-invariant measure $\mu$ is given by  
\begin{displaymath}
    h_{\mu}(f)=\sup\lbrace h_{\mu}(f,\mathcal{P}): \mathcal{P}\text{ is a finite and measurable partition of } M\rbrace, 
\end{displaymath}
where \begin{displaymath}
    h_{\mu}(f,\mathcal{P})=-\lim_{n\to\infty}\frac{1}{n}\sum_{P\in\mathcal{P}_n}\mu(P)\log\mu(P),\quad \mathcal{P}_n=\mathcal{P}_0\vee f^{-1}(\mathcal{P}_1)\vee\cdots\vee f^{n-1}(\mathcal{P}).
\end{displaymath}

The celebrated variational principle relates the metric entropies of a continuous map  and its topological entropy: 

\begin{theorem}[Variational Principle \cite{Katok}]
For any continuous map $f$ it holds:
$$h(f)=\sup\{h_{\mu}(f); \mu\in \mathcal{M}(f)\}=\sup\{h_{\mu}(f);{\mu\in \mathcal{M}_e(f)}\}  $$
\end{theorem}
\vspace{0.1in}
\noindent \textit{Smooth Dynamics} 
\vspace{0.1in}

In this part, we recall some well-known concepts from the theory of $C^1$-diffeomorphisms that will be useful to our research. Let us begin by recalling the concept of hyperbolic sets which was coined in the seminal work of S. Smale \cite{SM}.

\begin{definition}
    We say that a compact and invariant subset $\Lambda\subset M$
 is hyperbolic if there are $C>0$, $0<\lambda<1$, and a decomposition $T\Lambda=E^s\oplus E^u$ such that:
 \begin{itemize}
     \item  $Df_x(E^s_x)\subset E^s_{f(x)}$ and $Df^{-1}_{f(x)}(E^u_{x})$, for every $x\in \Lambda$.
     \item $|| Df^n|_{E^s}||\leq C\lambda^{-n}$, for every $n\geq 0$.
     \item $|| Df^{n}|_{E^u}||\leq C\lambda^n$, for every $n\leq 0$.
 \end{itemize}
 
 \end{definition}

Recall that a compact and invariant set $\Lambda$ is locally maximal if there is an  isolating neighborhood $U$ of $\Lambda$, i.e.,  $f(\overline{U})\subset U $ and $$ \Lambda=\bigcap_{n\geq 0} f^n(U).$$

\begin{definition}
    A subset $\Lambda\subset M$ is a horseshoe if it is hyperbolic, locally maximal, totally
disconnected, contains a transitive point, and it is not a periodic orbit.  
\end{definition}

The most famous example of the horseshoe is the Smale's horseshoe, which is widely known as one remarkable example of a chaotic dynamical system. The next proposition is a summary of well-known properties of horseshoes, that will be used by us. The reader interested in its proof can refer to \cite{Katok}, which contains a brilliant exposition of classical concepts of the dynamical systems theory. 
\begin{proposition}[\cite{Katok}]\label{hsprop}
    If $\Lambda\subset M$ is a horseshoe for $f$, then it holds:
\begin{itemize}
    \item $f|_{\Lambda}$ is sensitive and not minimal.
    \item $f|_{\Lambda}$ has the shadowing property.
    \item $\Lambda$ is $C^1$-robust, i.e., there is an isolating neighboorhood $U$ of $\Lambda$ and a neighborhood $\mathcal{U}\subset \text{Diff}^1(M)$  of $f$, such that if $g\in \mathcal{U}$, then there is a horseshoe $\Lambda_g\subset U$ for $g$.  
\end{itemize}    
\end{proposition}

\subsection{Smooth Flows}\label{singularsups}
We begin by recalling some basic facts from Flow's theory.
 
\begin{definition}
    A $C^1$-flow on $M$ is a $C^1$-map $\phi:\mathbb{R}\times M\to M$ satisfying:
    \begin{enumerate}
        \item $\phi(0,x)=x$, for every $x\in M$.
        \item $\phi(t+s,x)=\phi(t,\phi(s,x))$, for any $x\in M$ and any $t,s\in \mathbb{R}$.
    \end{enumerate}
    We denote by $\phi_t$ the time $t$ map $\phi(t,\cdot)$. 

\end{definition}
\begin{remark}
    A flow on $M$ can be seen as a family $\{\phi_t\}_{t\in \mathbb{R}}$ of $C^1$-diffeomorphisms such that $\phi_0=Id_M$ and $\phi_{t}\circ\phi_{s}=\phi_{t+s}$, for every $t,s\in \mathbb{R}$.
\end{remark}
An equivalent way of defining a $C^1$-flow is through vector fields. Indeed, if $X$ is a $C^1$-vector field over $M$, then $X$ induces on $M$ a flow whose velocity vector field is $X$. Conversely, to any $C^1$-flow  on $M$, one can assign a velocity vector field $X$. Thus, in this work, we will always denote by $X_{\phi}$  the velocity vector field that induces $\phi$. 
We say that a subset $K\subset M$ is \textit{$\phi$-invariant} if $\phi_t(K)=K$, for every $t\in \mathbb{R}$. A point $\sigma\in M$ is a \textit{singularity} for $\phi$ if $X_{\phi}(\sigma)=0$. We denote the singularities of $\phi$ by $Sing(\phi)$. A point $x\notin Sing(\phi)$ is said to be a \textit{regular point}. A regular point is \textit{periodic} for $\phi$ if, there is $\eta>0$ such that $\phi_{\eta}(x)=x$. We say that a flow is a \textit{regular flow} if its set of singularities is empty.

\begin{definition}
    The topological entropy of a flow $\phi$ is the topological  entropy of its time-one map, i.e., $$h(\phi)=h(\phi_1).$$ 
\end{definition}

Next, we shall discuss some properties of the suspension flows introduced in Section \ref{Intro}. The first main property is that two flows generated by suspending a diffeomorphism  $f$ through two distinct roof functions are equivalent. Let us precise this equivalence concept:

\begin{definition}
 Let $\phi$ and $\phi'$ be two flows on $M$ and $M'$, respectively. We say that $\phi$ and $\phi'$ are equivalent if there is a homeomorphism $h\colon M\to M' $ such that:
    \begin{enumerate}
    \item $h(O_{\phi}(x))=O_{\phi'}(h(x))$, for every $x\in M$. 
    \item  $h$ preserves the natural orientation induced on $O(x)$ by $\phi$, for every $x\in M$.
    \end{enumerate}
\end{definition}

In \cite{SV} it is proved that, for regular flows, the positiveness of topological entropy is preserved under equivalence. On the other hand, any pair of suspension flows of a diffeomorphism $f$ with distinct roofs $\rho$ and $\rho'$ are equivalent  (see \cite{BW}). Since our primary purpose in this work is to study the positiveness of entropy, to simplify our proof we will always consider suspensions whose roof functions are the constant function  $\rho\equiv 1$. In this case, given an invariant measure $\mu$ for $f$, we can   construct an  invariant measure $\overline{\mu}$ for $\phi$ by defining:
$$\int\xi d\overline{\mu}=\int\int_{0}^1\xi(\phi_t(\pi(\overline{x})))dtd\mu,   $$
for every $\xi\in C^0(\overline{M})$.

\begin{theorem}[\cite{SYZ}]
    Let $\phi$ be a suspension flow of $f$ and $\mu$ an invariant measure for $f$.
    Then $\overline{\mu}$ is ergodic if, and only if, $\mu$ is ergodic.  In addition, $h_{\mu}(f)=h_{\overline{\mu}}(\phi)$.
    \end{theorem}

\subsection{Singular Suspensions}

The process of suspension described earlier gives rise to a regular flow. Next, we will perform an extra process to obtain singular versions of suspension flows. Let $f: M\to M$ be a $C^1$-diffeomophism, and let $\phi$ be the suspension flow of $f$. Fix a compact set $S\subset \overline{M}$ with zero Lebesgue measure. Let $$\alpha:\overline{M}\to [0,+\infty)$$ be a smooth function such that $\alpha(x)=0$, if and only if $x\in S$. In this way, it is clear that $\alpha X$ is a $C^1$ vector field over $\overline{M}$. 

\begin{definition}
    The singular suspension of $f$ with roof $\rho$ and brake $\alpha$  is the flow $\psi:\mathbb{R}\times \overline{M}\to \overline{M}$ induced by the vector field $\alpha X$.
\end{definition}

 We would like to remark the following immediate consequence from the previous  definition:

\begin{enumerate}
    \item $Sing(\psi)=S$
    \item Any point  contained in $$\overline{M}\setminus\bigcup_{t\in \mathbb{R}}\phi_t(S)$$ keeps the same orbit under $\psi$ as for $\phi$.
    \item For any regular point $$x\in \bigcup_{t\in \mathbb{R}}\phi_t(S),$$  there is $\sigma_{x}\in S$ such that either  $$\lim\limits_{t\to \infty}\psi_{t}(x)=\sigma_x \textrm{ or } \lim\limits_{t\to -\infty}\psi_{t}(x)=\sigma_x.$$  
    
\end{enumerate} 

Let us go further and see how the function $\alpha$ modifies the dynamics of $\phi$ in terms of 
time-change. The reader interested in the proof of the previous facts is recommended to refer to \cite{To}. Let  $m$ denote the Lebesgue measure of $\overline{M}$.  Since $\alpha$ is continuous, non-negative, bounded and $m(S)=0$ we have $$\int \alpha dm>0. $$ We then define the following function:
$$\theta(t,\overline{x}):=\int_0^t \alpha(\phi_s(\overline{x}))ds,$$
where $\overline{x}\in M\times \{0\}$. The flow $\psi$  induced by $\alpha X$ satisfies the following relation:
$$\psi_t(\overline{x})=\phi_{\tau(t,\overline{x})}(\overline{x}),$$
where  $$\tau(t,\overline{x})=\sup\{s\geq0; \theta(s,\overline{x})\leq t\}.$$ Recall that the fiber $M\times\{0\}\subset \overline{M}$ is naturally identified with $M$. Let $$\pi:\overline{M}\to M\times\{0\}$$ denote the natural projection on the first coordinate.  Next, define a function $\gamma: M\to [0,+\infty] $ by :

$$\begin{cases} 
\gamma(x)=\infty, \textrm{ if  }  x\in \pi(S)\\ 
    \psi_{\gamma(x)}((x,0))=\phi_1((x,0)), \textrm{ otherwise}. 
\end{cases}$$

The idea behind the function $\gamma$ is quite simple. It gives us the first time such that a point $\overline{x}=(x,0)$ hits $(f(x),0)$ for any $\overline{x}$ which is not contained in a  vertical segment containing a singularity. Since $\psi$ is smooth, then $\gamma$  is clearly continuous and $$\gamma(x_n)\to \infty, \textrm{ when } x_n\to x\in \pi(S).$$ This function will be crucial to obtain our results. We end this section by setting another important notation. Namely, let us denote 
$$A_{Sing}=\pi\left (\bigcup_{t\in \mathbb{R}}\phi_t(S)\right).$$
Observe that $A_{Sing}$ is formed all the points in $x\in M$ such that there is $n\in \mathbb{Z}$ satisfying $(f^n(x),0)$ is contained in a vertical segment containing a singularity of $\psi$.

\section{Entropy of Singular suspensions}\label{characsection}

In this section, we begin the study of the entropy of singular suspensions. Our first result tells us that it is pretty easy to destroy the topological entropy of a suspension if one is allowed to insert a ``big set" of zeros for a break $\alpha$. 

\begin{rmk}\label{uncountable}
     Notice that for any diffeomorphism $f$ with positive entropy, there exists a brake $\alpha$ such that the suspension flow $\psi$  has zero topological entropy. Indeed,    Let $f: M\to M$ be diffeomorphism with positive topological entropy. Let $\phi$ be the suspension flow of $f$ on  $\overline{M}$. Fix the fiber $S=M\times \{\frac{1}{2}\}$. Thus, $S$ is a compact set satisfying $m(S)=0$. Now, take any smooth function $\alpha:M\to [0,\infty)$ such that $\alpha(x)=0$ if, and only if $x\in S$. Consider the flow $\psi$ generated by $\alpha X$. By construction, for  any regular point $\overline{x}\in \overline{M}$ there are $\sigma_+,\sigma_-\in S$ such that 
         $$\psi_t(\overline{x})\to\sigma_+  \,\,\ \textrm{and } \,\,\  \psi_{-t}(\overline{x})\to \sigma_-,$$
         when $t\to \infty$. Therefore, the non-wandering set of $\psi$ is formed by $S$ and therefore  $h(\psi)=0$ (see \cite{Katok}).  
\end{rmk}

From the previous theorem, we can see that we need to impose some restrictions on the set $S$  to investigate the topological entropy of singular suspension flows. In \cite{SYZ}, it was exhibited a singular suspension flow with a single singularity, zero topological entropy, but whose basis is a minimal diffeomorphism with has positive topological entropy. We will see that this process is uncommon, even in cases where $S$ is infinite. 

The next result gives us a full characterization of suspensions with zero topological entropy in terms of the function $\gamma$ defined in Subsection \ref{singularsups}. To simplify notation, for every Borelian measure $\mu$ on $M$ and measurable function $\xi:M\to \mathbb{R}$, denote
$$E_{\mu}(\xi)=\int \xi d\mu . $$

\begin{theorem}\label{equiv}
    Let $f: M\to M$ be a diffeomorphism with positive topological entropy and let $\psi$ be a singular suspension of $f$ with a break  $\alpha$. Then $h(\psi)=0$ if, and only if, for every ergodic measure for $f$ satisfying $h_{\mu}(f)>0$, we have $E_{\mu}( \gamma) =+\infty$.  
\end{theorem}

\begin{proof}
   First, suppose    $E_{\mu}( \gamma)=+\infty$ for every ergodic measure of $f$ with positive entropy. By the variational principle, to show that $h(\psi)=0$ we just need to show that any non-atomic ergodic measure $\overline{\mu}$ for $\psi$ satisfies $h_{\overline{\mu}}(\psi)=0$. Let $\overline{\mu}$ be an ergodic and non-atomic measure for $\psi$. By in \cite[  Lemma 2.11]{SYZ}, there is a unique non-atomic ergodic measure $\mu$ for $f$ such that:    
    $$\int \xi d\overline{\mu}=\frac{1}{E_{\mu}( \gamma)}\int \int_0^{\gamma(\pi(\overline{x}))}\xi(\psi_t((\pi(\overline{x}))dt d\mu,  $$
    for every measurable function  $\xi$.
    Now, we split the proof into two cases:\\
    \textbf{Case 1:} Suppose $E_{\mu}(\gamma)=+\infty$. Then,  by following the steps of the proof of   \cite[Corollary 2.12]{SYZ} one can conclude that $\overline{\mu}$  is atomic, a contradiction.\\
    \textbf{Case 2:} Suppose $E_{\mu}(\gamma)<+ \infty$. In particular, by assumption, we have $h_{\mu}(f)=0$.
     Now, observe:
     $$\overline{\mu}(B)=\frac{1}{E_{\mu}( \gamma)}\int_{B'} \int_0^{\gamma(\pi(\overline{x}))}1(\psi_t(\pi(\overline{x}))dt d\mu=\frac{1}{E_{\mu}( \gamma)}\int_{B'}\int_0^{1}\gamma(\pi(\overline{x}))dt d\mu=\int_{B'} \gamma d\overline{\mu},$$
     where $B'=\pi(B)$. Then,  \cite[Theorem 2.9]{SYZ} implies $h_{\overline{\mu}}(\psi)=0$.

    Conversely, suppose there is an ergodic measure $\mu$ for $f$ such that $h_{\mu}(f)>0$ and $E_{\mu}( \gamma)<+\infty$. 
    Define a measure $\overline{\mu}$ on $\overline{M}$ by
    $$\int \xi d\overline{\mu}=\frac{1}{E_{\mu}( \gamma)}\int \int_0^{\gamma(\pi(\overline{x}))}\xi(\psi_t((\pi(\overline{x}))dt d\mu,  $$
    for every measurable function $\xi$. If we can prove that $\overline{\mu}$ is an ergodic and invariant measure for $\psi$, then  \cite[Theorem 2.9]{SYZ} will imply $h(\psi)>0$.
    
    Thus, we finish this proof showing that $\overline{\mu}$ is an ergodic invariant measure. To see that $\overline{\mu}$ is a probability, notice 
  $$\overline{\mu}(\overline{M})=\frac{1}{E_{\mu}(\gamma)}\int \int_0^{\gamma(\pi(\overline{x}))} 1 dtd\mu=  1.$$ 
  
  Next, we check that $\overline{\mu}$ if $\psi$-invariant. Since the proof for negative time is completely analogous, we will prove the invariance of $\overline{\mu}$ only for positive time. Fix $s>0$. Since we suppose $S$ is countable, then $A_{Sing}$ is countable. On the other hand, $h_{\mu}(f)>0$ implies $\mu$ is non-atomic. Therefore, $\mu(A_{Sing})=0$. We thus conclude $$\overline{\mu}(\mathcal{O}_{\psi}(A_{Sing}))=0,$$ 
  where $\mathcal{O}_{\psi}(K)$ denotes the orbit of a subset of $K\subset \overline{M}$ under $\psi$.
  
  For each $x\in M\setminus A_{sing} $, we denote by $I_x$ the vertical segment ${\psi}_{[0,\gamma(\pi(x)))}(x)$.  Fix a measurable set $A\subset M$. For each $k>0$ denote: 
  $$A_k=\{x\in A\setminus \mathcal{O}_{\psi}(A_{Sing}); \psi_s(x)\in I_{f^k(\pi(x))} \}.$$
Notice that the sets $A_k$ are measurable and pairwise disjoint.  In particular,  $$\overline{\mu}(\psi_{s}(A))=\overline{\mu}({\phi}_{s}(A)\setminus \mathcal{O}_{\psi}(A_{Sing}))=\Sigma_{i=1}^{\infty} \overline{\mu}(\psi_{s}(A_i).$$
But, $$\overline{\mu}(A_i)=\frac{1}{E_{\mu}(\gamma)}\int_{\pi(A_i)} \int_0^{\gamma(\pi(\overline{x}))} 1(\psi_t(\pi(\overline{x}))) dtd\mu.$$
Therefore the invariance of $\mu$ implies $$\overline{\mu}(A_i)=\frac{1}{E_{\mu}(\gamma)}\int_{f^i(\pi(A_i))} \int_0^{\gamma(\pi(\psi_s(\overline{x})))} 1(\pi(\psi_s(\overline{x}))) dtd\mu=\overline{\mu}(\psi_s(A_i)).$$  
We then conclude 
$$\overline{\mu}(A)=\overline{\mu}(A\setminus \mathcal{O}_{\psi}(A_{Sing}))=\sum_{i=1}^{\infty } \overline{\mu}(A_i)=\overline{\mu}(\psi_s(A)).$$
  
Finally, to see that $\mu$ ergodic  we consider $A\subset \overline{M} $  a $\psi$-invariant  set, i.e., $\psi_t(A)=A$ for every $t\in \mathbb{R}$. It is easy to check that the set $\pi(A)\setminus A_{Sing}$ is $f$-invariat. 
Now the ergodicity of $\mu$ implies $\mu(\pi(A)\setminus A_{Sing}))$ is either equal to $0$ or $1$.  Since  $\overline{\mu}(A\setminus O_{\phi}(A_{Sing}))=\overline{\mu}(A)$, we obtain by definition that $\overline{\mu}(A)$ is either $0$ or $1$ and the proof is complete.

 \end{proof}

The next result is contained in the proof Theorem \ref{equiv}. We decided to state it explicitly to recall it when necessary in the remainder of this work. \\

\begin{lemma}\label{zeroentr}
    Let $\psi$ be a singular suspension of $f$. If $h(f)=0$, then $h(\psi)=0$. 
\end{lemma}

Under the light of the previous results, we remark that to prove Theorems \ref{generic} and \ref{3d} we just need to show that if $f$ has positive entropy, then $\psi$ has positive entropy.
In addition, as a corollary of the proof of Theorem \ref{equiv}, we obtain a useful criterion for positive entropy of singular suspensions, even with uncountable many singularities.

\begin{theorem}\label{extraresult}
    Let $f$ be a diffeomorphism on $M$ and $\psi$ be a singular suspension of $f$. Suppose there is a non-atomic ergodic measure $\mu$ such that $h_{\mu}(f)>0$, and $Supp(\mu)\cap A_{Sing}=\emptyset$. Then $h(\psi)>0$. 
\end{theorem}
\begin{proof}
    The result is derived by following the arguments of the proof of Theorem \ref{equiv} and observing:
    \begin{enumerate}
    
    \item  $\mu(A_{Sing})=0$.
    \item  $E_{\mu}(\gamma)<+\infty$,  since $\gamma$ is continuous on $M\setminus A_{Sing}$. 
    \end{enumerate}
\end{proof}

\section{Proof of Theorems \ref{generic}, \ref{generic2} and \ref{3d}.}\label{secproofabc}

In this section we give the proofs for Theorems \ref{generic}, \ref{generic2} and \ref{3d}. Let us begin by collecting some ingredients for our proofs. The first one the following result:

\begin{theorem}[Theorem 5.4 in \cite{MO}]\label{potencia}
Let $K$ be a metric space and $f: K\to K$ be a non-wandering expansive homeomorphism with the shadowing property. For every non-empty open $U\subset K$, there are $n>0$ and a compact and $f^n$-invariant subset $K_U\subset U$ such that $f^n|_{K_U}$ is topologically conjugated to the full shift map of two symbols $(\Sigma,\sigma)$.
\end{theorem}

The next result is one of the key tools used in the proof of our main results.

\begin{theorem}\label{entrhs}
Let $f:M\to M$ be a $C^1$ diffeomorphism. Suppose $M$ contains a horseshoe. If $\psi$ is a singular suspension of $f$ with countable many singularities, then $h(\psi)>0$.
\end{theorem}
\begin{proof}
 Under the light of Theorem \ref{equiv}, to show  $h(\psi)>0$, we need to find an ergodic measure for $f$ with $E_{\mu}(\gamma)<+\infty$. On the other hand, since $\gamma$ is continuous in  $M\setminus \pi(S)$, if we find an ergodic measure for $f$ with positive entropy and whose support is contained in $M\setminus A_{Sing}$ we are done.  For this sake, all our task is reduced to prove the following: \\
\textbf{Claim:} There is a compact and invariant set $K^{*}$ with positive topological entropy such that 
$$K^{*}\cap A_{Sing}=\emptyset.$$
\begin{proof}[\emph{\textbf{Proof of Claim}}]
Assume that $\Gamma$ is a horseshoe for $f$. It is well known that the hyperbolicity and local-maximality of $\Gamma$ implies $f|_{\Gamma}$ is expansive and has the shadowing property. Therefore, Theorem \ref{potencia} implies the existence of $n>0$ and $K\subset \Gamma$ so that $f^n|_{K}$ is topologically conjugated to the full shift of two symbols $\sigma$. Now by \cite[Theorem 2.5]{Gr}, for every $c\in [0,h(\sigma))$, there is a minimal subset $\Lambda_c\subset \Sigma$ such that $h(\sigma|_{\Lambda_c})=c$. Since $(0,h(\sigma))$ is uncountable, we obtain uncountable many minimal subsets $\Lambda_c$. In addition, those sets must be pairwise disjoint. Otherwise, there would exist  $c>c'>0$ and minimal subsets $\Lambda_c,\Lambda_{c'}$ satisfying $\Lambda_c\cap\Lambda_{c'}\neq\emptyset$. So, the minimality assumption implies $\Lambda_c=\Lambda_{c'}$ and this is impossible because we would have a minimal set with two distinct values for its entropy.  

\noindent Let $g:K\to \Sigma$ be the conjugacy homeomorphism between $f^n$ and $\sigma$ and denote $$K_{c}=\bigcup_{i=0}^n f^i(g^{-1}(\Lambda_c)).$$
Notice that each $K_c$ is minimal and since $h(f|_{K_c})=\frac{c}{n}$, we also have that the family $\{K_c\}_{c\in (0,h(\sigma))}$ is uncountable. 

To conclude the proof of the claim, we observe that since $A_{Sing}$ is countable, then it must intersect at most countable many sets $K_c$ and therefore, there is $K^*$ minimal, with positive entropy and disjoint from $A_{Sing}$.

\end{proof}
Finally, to conclude the proof we fix a minimal set $K^*$ with positive topological entropy and disjoint from $A_{Sing}$ as in the claim. Choose any ergodic measure $\mu$ supported in $K$ and with positive entropy. since the function $\gamma$ is  continuous on $M\setminus A_{Sing}$, then $\gamma$ is integrable on $Supp(\mu)$, therefore the conclusion is derived from Theorem \ref{equiv}.
\end{proof}

We are now in position to provide a proof for our main results.

\begin{proof}[Proof of Theorem \ref{generic}]
Let $M$ be a closed manifold. By  \cite[Theorem A]{Cro} there is a residual subset  $\mathcal{R}\subset \text{Diff}^1(M)$,   such that if $f\in \mathcal{R}$, then $f$ satisfy one of the following conditions:

\begin{enumerate}
    \item $f$ is a Morse-Smale diffeomorphism. 
    \item $f$ contains a homoclinic transverse intersection.
    \item $f$ contains a minimal subset $K_0$, which is the Hausdorff limit of non-trivial homoclinic classes. 
\end{enumerate}

 We will now show that this is the desired residual set.\\ 
 \textbf{Case 1:} Suppose $f$ is Morse-Smale, then $h(f)=0$. In addition, Morse-Smale diffeomorphisms are structurally stable. In particular, there is a neighborhood $\mathcal{U}$ of $f$ such that $g\in \mathcal{U}$, then $h(g)=0$ (see \cite{Katok}). Therefore, the conclusion is derived from Theorem \ref{zeroentr}.\\ 
 \textbf{Case 2:} If $f$ 
 contains a homoclinic transverse intersection, it contains a horseshoe (see \cite{Katok}). Since horseshoes are robust,  Theorem \ref{entrhs}  $h(\psi)>0$, if $\psi$  is the suspension of any diffeomorphism $C^1$-close to $f$. \\
 \textbf{Case 3:} Finally, if $f$ contains a minimal subset $K_0$, which is the Hausdorff limit of non-trivial homoclinic classes, then obviously $f$ contains a homoclinic class. Since homoclinic classes contain horseshoes, we derive the result exactly as in Case 2.
\end{proof}
The proof of Theorem \ref{generic2} is more delicate and requires the construction of a horseshoe for an open and dense set of conservative $C^1$ diffeomorphisms. Denote $\text{Diff}^1_{m, DS}(M)$ the subset of $\text{Diff}^1_{m}(M)$ with dominated splitting and $\mathcal{E}_{m}(M)$ the interior of $\text{Diff}^1_{m}(M)\setminus \text{Diff}^1_{m, DS}(M)$. 
In \cite[Theorem 1]{BCT} it was proved that there is a generic $\mathcal{G}_1\in \mathcal{E}_{m}(M)$ such that for every $f\in \mathcal{G}_1$ 
\begin{equation*}\label{positive_entropy}
    h_{top}(f)=\sup_{K\in H(f)}h_{top}(f,K)>0,
\end{equation*}
where $$H(f):=\{K\subset M: K  \,\, \text{is a horseshoe of}\,\, f \}.$$
To fix ideas, as always for generic sets,  there is a countable family of open and dense sets  $V_i$ in $\mathcal{E}_{m}(M)$  such that 
$$\mathcal{G}_1=\bigcap_{i=1}^{\infty}V_i.$$
Also, in \cite[Corollary D]{ACW2}  Avila, Crovisier, and Wilkinson proved that there is a dense $\mathcal{G}_2\subset \text{Diff}^1_{m, DS}(M)$ such that each $f\in \mathcal{G}_2$ has a horseshoe.

\begin{proof}[\emph{\textbf{Proof of Theorem \ref{generic2}}}]
First, we consider the set $\mathcal{G}_3:=\mathcal{G}_2\cup \mathcal{G}_1$. By definition, $\mathcal{G}_3$ is a dense set. Moreover, for each $f\in \mathcal{G}_3$ we have a horseshoe as we wish. Hyperbolic sets are persistent after $C^1$ perturbation (see Proposition \ref{hsprop}). So, for each $f\in \mathcal{G}_3$ there is a $C^1$-neighborhood $\mathcal{U}_f$ of $f$ such that, all $g\in \mathcal{U}_f$ has a horseshoe. Thus, consider the set 
$$\mathcal{R}:=\bigcup_{f\in \mathcal{G}_3}\mathcal{U}_f,$$
which is dense and open. Finally, observe that, by construction, for $f\in \mathcal{R}$ we have horseshoe, and consequently we obtain $h(\psi)>0$ from Thereom \ref{entrhs}.

\end{proof}

\begin{proof}[Proof of Theorem \ref{3d}]
Let $\phi$  be a $C^{1+\alpha}$-suspension flow of a diffeomorphism $f$. Therefore, $f$ is a $C^{1+\alpha}$-diffeomorphism of a surface $M$.  

Suppose $f$  has positive topological entropy. Therefore,  \cite[Corollary 4.3]{Kat2} implies the existence of a topological horseshoe with positive entropy. Since by assumption $Sing(\psi)$ is countable, we obtain $h(\psi)>0$ from Theorem \ref{entrhs}. 
    
\end{proof}
Observe that the previous result is false in higher dimensions. Indeed, the example introduced in \cite{SYZ} is the suspension of a $C^{\infty}$-diffeomorphism on a four-dimensional manifold with positive entropy and presents a singular suspension of zero entropy. Nevertheless, our techniques can be further applied to obtain a similar result for the higher dimensional setting, if we suppose an extra hypothesis. Precisely, we have the following result:

\begin{theorem}
 Let $\psi$ be a  $C^{1+\epsilon}$-flow obtained from the singular suspension of a diffeomorphism $f$. If $Sing(\psi)$  is countable and $f$ admits a hyperbolic measure with positive entropy, then $\psi$ has positive topological entropy.
\end{theorem}

\begin{proof}

The proof of this result is completely analogous to the proof of Theorem \ref{3d}. We just need to notice that the main tool to prove Theorem \ref{3d} was Katok's result in dimension $2$, which guarantees the existence of horseshoes  when the topological entropy is positive (cf. \cite{Kat2}). The same result holds under the assumption of existence of hyperbolic measures (cf. \cite{Kat2}). 
\end{proof}

\section{Expansive Flows}\label{expsection}

In this section, we exploit the entropy theory of expansive singular flows. The concept of expansive flow, introduced in \cite{BW}, was established as a mechanism to understand the expansive properties of hyperbolic flows. Over time, it developed into a rich, standard theory of expansiveness for non-singular flows. Nevertheless, it has been shown to be inefficient in dealing with flows displaying singularities accumulated by regular orbits. Despite this, some singular flows exhibit a behavior similar to the expansive ones but fail to be expansive as  \cite{BW}. The most famous example of such a flow is the geometric Lorenz attractor. To deal with the Lorenz attractor, it was introduced in \cite{Ko1}, the concept of $k^*$-expansive flow. Here, we will call both concepts expansive flow. This is supported by the following facts: The expansiveness in \cite{BW} is stronger than the $k^*$-expansiveness, and they coincide in the non-singular scenario (see \cite{Ko1}). Let us now recall its definition. Denote by $Rep$ the set of increasing homeomorphisms $h:\R\to \R$ satisfying $h(0)=0$.        

\begin{definition}
    A flow is expansive if there for every $\eps>0$ there is  $\delta>0$ such that if $x,y\in M$, $h\in Rep$ satisfy $d(\phi_t(x),\phi_{h(t)}(y))\leq \delta$, for every $t\in \mathbb{R}$ then there are $t_0\in \mathbb{R}$ such that $\phi_{h(t_0)}(y)\in\phi_{[t_0-\eps,t_0+\eps]}(x)$. 
\end{definition}

\begin{remark}
    It is well known that expansiveness is independent of the metric on $M$ (see \cite{BW} and \cite{Ko2}).
\end{remark}
In \cite{ACP}, it was proven that non-singular expansive flows have positive topological entropy if the ambient manifold has a dimension of at least two. Actually, the dimensional hypothesis can be replaced by dimension three, since there are not expansive non-singular flows on compact surfaces. So we may wonder if a similar result holds for expansive singular flows. The first important thing to notice is that there are examples of expansive singular flows on surfaces (\cite{Ar3}). Moreover, surface flows have zero topological entropy. Because of this, we need to assume $M$ to be at least three-dimensional. In this section, we will answer this question for expansive suspension flows over three-dimensional manifolds. 

We begin by recalling the following elementary property of singular expansive flows. We decided to provide proof for the reader's convenience.  

\begin{proposition}\label{finitesing}
    If a flow $\phi$ on $M$ is expansive, then $Sing(\phi)$ is finite. 
\end{proposition}
\begin{proof}
    Suppose $\phi$ is expansive, fix $\eps>0$ and let $\delta>0$ be given by the expansiveness of $\phi$. We claim that if $\sigma\in Sing(\phi)$, then $B_{\delta}(\sigma)\cap Sing(\phi)=\{\sigma\}$. Indeed, fix $\sigma,\sigma'\in Sing(\phi)$ and suppose $\sigma'\in B_{\delta}(\sigma)$. Thus, we have $d(\phi_t(\sigma),\phi_t(\sigma'))\leq\delta$, for every $t\in \mathbb{R}$. Therefore, the expansiveness implies $\sigma=\sigma'$, and the claim is proved. The proof ends by noticing that $Sing(\phi)$ is a compact set.  
\end{proof}

Next, we recall the concept of expansiveness in the discrete-time setting:
\begin{definition}
    We say that $f:M\to M$ is expansive if there is $e>0$ such that $B^{\infty}_{e}(x)=\{x\}$, for every $x\in M$, where $$B^{\infty}_{e}(x)=\{y\in M; d(f^n(x),f^n(y))\leq e, \forall n\in \mathbb{Z}\}.$$ The number $e$ is called the expansiveness constant of $f$. 
\end{definition}

Naturally, the notion of expansiveness deals with metric spaces. At this point, we need to define the appropriate metric for our manifold $\overline{M}$ (see the comment before definition \ref{Def-Susp} ).  The metric $\overline{d}$ on $\overline{M}$ is defined as follows (see \cite{BW}). First, for any two points $\overline{x}$ and $\overline{y}$ in a fiber $M\times\{s\}$ define: $$\overline{d}(\overline{x},\overline{y})=sd(\pi(\overline{x}),\pi(\overline{y}))+(1-s)d(f(\pi(\overline{x})),f(\pi(\overline{y}))).$$
Secondly, if $\overline{x}$ and $\overline{y}$ belongs to the same vertical segment, let $\overline{d}(\overline{x},\overline{y})$ denote the shortest distance between $\overline{x}$ and $\overline{y}$ along the $\phi$-orbit of $\overline{x}$, where the distance is given by the standard distance on $\mathbb{R}$.

Now,  let $CH(\overline{x},\overline{y})$ denote the family of finite chains $\{w_0,w_1,\dots,w_i\}$, where $w_0=\overline{x}$, $w_i=\overline{y}$ and for every $k=0,\dots,i-1$ one has that either $w_k$ and $w_{k+1}$ are contained in the same fiber or $w_k$ and $w_{k+1}$ are contained in the same vertical segment. Finally , the metric $\overline{d}$ is defined as 
$$\overline{d}(\overline{x},\overline{y})=\inf\left\{\sum_{k=0}^{i-1}\overline{d}(w_i,w_{i+1}); \{w_0,\dots,w_i\}\in CH(\overline{x},\overline{y})\right\}.$$

Now, we slightly improve \cite[Theorem 3.6]{Shi} to obtain several examples of expansive singular suspensions. 

\begin{theorem}\label{expsuspension}
    Let  $f$ be an expansive diffeomorphism. If a brake $\alpha$ has finitely many zeros, then the singular suspension $\psi$ with brake $\alpha$ is expansive.
\end{theorem}

\begin{proof}

Since $\alpha$ has finite zeros, then $Sing(\psi)$ is finite.  Without loss of generality, we can assume $Sing(\psi)\cap (M\times \{0\})=\emptyset$.
Denote by $e$ the expansiveness constant of $f$ and 
$$d=\min\{\overline{d}(\sigma_i,\sigma_j); \sigma_i,\sigma_j\in Sing(\psi), \sigma_i\neq\sigma_j\}.$$

Fix $\eps>0$  such that $M_{\eps}\cap Sing(\psi)=\emptyset$, where $$M_{\eps}=\psi_{[-\eps,\eps]}(M\times\{0\})$$
\begin{center}
    and 
\end{center} $$\text{diam}(\psi_{[-\eps,\eps]}(\overline{x}))\leq\frac{e}{2},$$
for every $\overline{x}\in M\times\{0\}$.
Now, fix $$4\delta\leq\min\{d,e\}$$ such that $B_{4\delta}(\overline{x})\subset M_{\eps}$, for every $\overline{x}\in M\times\{0\}$.

Fix $\overline{x},\overline{y}\in \overline{M}$ and suppose there is a reparametrization $h$ such that $$\overline{d}(\psi_{t}(\overline{x}),\psi_{h(t)}(\overline{y}))\leq\delta, $$
for every $t\in \mathbb{R}$.\\

First, notice that the choice of $\delta$ implies that $\overline{x}$ and $\overline{y}$ are regular. Now, suppose $\overline{y}\in \mathcal{O}_{\psi}(\overline{x})$. In this case, there is $t\in  \mathbb{R}$ such that $\psi_{t}(\overline{x})\in M\times\{0\}$. In this case $\psi_{h(t)}\in M_{\eps}$ and therefore $\overline{y}\in \psi_{[-\eps,\eps]}(\overline{x})$.

Next, observe that  by the choice of $\delta$, if $\overline{y}\notin \mathcal{O}_{\psi}(\overline{x})$, then  $\pi(\overline{x}),\pi(\overline{y})\notin A_{Sing}$. Indeed, if this is the case,  there should be $\sigma_x$ and $\sigma_y\in Sing(\psi)$ such that 
$$\psi_{t}(\overline{x})\to \sigma_x \textrm{ and } \psi_{h(t)}(\overline{y})\to \sigma_y,$$
when $t\to \infty$ or $t\to-\infty$. contradicting the choice of $\delta$. 

For each $n\in \mathbb{Z}$, let $t_n\in \mathbb{R}$ be such that $$\psi_{t_n}(\overline{x})=(f^n(\pi(\overline{x})),0)$$ \begin{center}
    and
\end{center} $$\psi_{(t_n,t_{n+1})}(\overline{x})\cap (M\times \{0\})=\emptyset.$$ 

Since $$\overline{d}(\psi_{t_n}(\overline{x}),\psi_{h(t_n)}(\overline{y}))\leq\delta,$$
for every $n\in \mathbb{Z}$, there is $-\eps<\eps_n<\eps$  such that $$\psi_{h(t_n)+\eps_n}(\overline{y})=(f^n(\pi(\overline{y})),0).$$
Therefore, from definition of $\overline{d}$ 
\begin{eqnarray*}
 d(f^n(\pi(\overline{x})),f^n(\pi(\overline{y})))&=& \overline{d}(\psi_{h(t_n)}(\overline{x}), \psi_{h(t_n)+\eps_n}(\overline{y}))\\
 &\leq &  \overline{d}(\psi_{h(t_n)}(\overline{x}), \psi_{h(t_n)}(\overline{y}))+  \overline{d}(\psi_{h(t_n)}(\overline{y}), \psi_{h(t_n)+\eps_n}(\overline{y})),\\
 &\leq&\delta+\frac{e}{2}< e
\end{eqnarray*}
for every $n\in \mathbb{Z}$. So, the expansiveness of $f$ implies $\pi(\overline{x})=\pi(\overline{y})$ and therefore $\psi_{h(t_0)+\eps_0}(\overline{y})=\psi_{t_0}(\overline{x})$ and this completes the proof.

\end{proof}

\begin{remark}
    The main difference between the previous result and \cite[Theorem 3.6]{Shi}, is that while in \cite{Shi} it is assumed that $\alpha$ has one zero, here we assume that $\alpha$ can have how many zeros we want, once they are finitely many.
\end{remark}

We now proceed to obtain the proof of Theorem \ref{expentr}.  The first step towards its proof is the following result that provides us with a converse for Theorem \ref{expsuspension}.

\begin{theorem}\label{cw-exp}
    Let  $\psi$ be a singular suspension of $f\colon M\to M$ with a brake $\alpha$. If $\psi$ is expansive, then $f$ is expansive.    
\end{theorem}
\begin{proof}
    Suppose $\psi$ is expansive. We need some previous notation to simplify the exposition during this proof. First, we set the notations $d$ and $\overline{d}$ for the metrics on $M$ and $\overline{M}$, respectively. Let us denote $\overline{x}=(x,0)\in \overline{M}$. By Proposition \ref{finitesing}, $Sing(\psi)$ is finite, and therefore, $A_{Sing}$ is countable. 
    
    Fix $\eps>0$ and consider $\delta>0$ given by the expansiveness of $\psi$. We stated that $e=\frac{\delta}{2}$ is a countable-expansiveness constant for $f$. To prove that, we shall prove that the dynamical ball $B^{\infty}_{e}(x)$ is countable for every $x\in M$. For this sake, we need the following:\\
    \textbf{Claim:} If $x,y \in M\setminus A_{Sing}$ and $y\in B^{\infty}_{\delta}(x)$, then $x=y$.
    \begin{proof}[\emph{\textbf{Proof of Claim}}]
       We denote $\overline{x}=(x,0)$ and $\overline{y}=(y,0)$. From hyphoteses  $$d(f^n(x),f^n(y))\leq\delta,$$ for every $n\in\mathbb{Z}$. Recall that for every $\overline{z}\in \overline{M}$ and every $n\in\mathbb{Z}$, define $t(\overline{z},n)$ by
    $$\psi_{t(\overline{z},n)}(\overline{z})=(f^n(\pi(\overline{z})),0).$$ 
    
    Next, we will construct a suitable reparametrization for $\overline{y}$. First, notice that for any $n\in \mathbb{Z}$ and any $t\in [t(\overline{x},n),t(\overline{x},n+1))$, there is a unique  $s(\overline{x},t)\in [0,1)$ such that 
    $$\psi_t(\overline{x})=\phi_{\tau(\overline{x},t)}(\overline{x})=(f^n(\pi(\overline{x})),s(\overline{x},t)).$$
    Let us define homeomorphism $h:\mathbb{R}\to \mathbb{R}$ as follows. 
     For any $n\in \mathbb{Z}$, and $$t\in [t(\overline{x},n),t(\overline{x},n+1))$$ choose $$h(t)\in[t(\overline{y},n),t(\overline{y},n+1))$$ as the unique time satisfying $$\psi_{h(t)}(\overline{y})\in M\times \{s(\overline{x},t)\}$$          In this way, for every $t\in \mathbb{R}$, one has that $\psi_t(\overline{x})$ and $\psi_{h(t)}(\overline{y})$ belong to the same fiber $\overline{M}\times\{s\}$. 
     Since $d(f^n(x),f^n(y))\leq \delta$, for every $n\in \mathbb{Z}$, by the definition of $\overline{d}$ and  we have:

     $$ \overline{d}(\psi_t(\overline{x}),\psi_t(\overline{y}))\leq s(\overline{x},t) d(f^n(x),f^n(y)))+ (1-s(\overline{x},t))(d(f^{n+1}(x),f^{n+1}(y)))\leq \delta,$$     
     fore every $t\in [t(\overline{x},n),t(\overline{x},n+1))$.
     
     The expansiveness of $\psi$ gives us $t_0\in \mathbb{R}$ such that
     $$\psi_{h(t_0)}(\overline{y})\in \psi_{[t_0-\eps,t_0+\eps]}(\overline{x}).$$ 
     By construction of $h$, $\phi_{t_0}(\overline{x})$ and $\phi_{h(t_0)}(\overline{y})$  are in the same fiber and therefore $$\phi_{t_0}(\overline{x})=\phi_{h(t_0)}(\overline{y}).$$ We then conclude $x=\pi(\overline{x})=\pi(\overline{y})=y$.   
    \end{proof}
To conclude our proof, we observe that the claim implies $f$ is expansive in $M\setminus A_{Sing}$. Moreover, by definition $A_{Sing}$ is a finite collection of orbits of $f$. So, we can apply \cite[Main Theorem]{Williams} and conclude that $f$ is  expansive, as we wished.

\end{proof}

\begin{proof}[Proof of Theorem \ref{expentr}]
    Let $\psi$ be an expansive $C^{1}$-flow on a three-dimensional manifold which is a singular suspension. Proposition \ref{finitesing} implies $Sing(\psi)$ is finite. By Theorem \ref{cw-exp}, $\psi$ is the suspension of an expansive  $C^1$-diffeomorphism $f$ on a boundary-less surface $S$.  Next, let us find a horseshoe for $f$. To do this, we recall Theorem \cite[Theorem 1.6]{Lewo}, which states that $f$ must be topologically conjugated to a pseudo-Anosov diffeomorphism $g$. On the other hand, every pseudo-Anosov surface diffeomorphism has, by definition, a pair of transverse stable and unstable foliations which are dense due \cite[Corollary 14.15]{FB}, since $S$ is boundary-less. In addition,  \cite[Theorem 14.19]{FB} implies $g$ has dense periodic orbits. Finally, for any periodic point  $p\in S$, the denseness of the stable and unstable manifolds of $p$ implies the existence of a homoclinic intersection and, therefore, a horseshoe for $g$. Since $g$ is conjugated to $f$, then $f$ also has a horseshoe, and therefore, Theorem \ref{entrhs} implies that $\psi$ has positive entropy.

\end{proof}

\section{Some Examples and Further Discussion}\label{exesection}

This Section is devoted to providing some examples and discusses some aspects of the strength of our results. Let us begin by recalling that in the examples constructed in  \cite{SYZ} $f$ is minimal and $\psi$ has a unique singularity. In our first example, we show a diffeomorphism with positive entropy for that is not minimal and it is impossible to obtain a singular suspension with zero entropy by adding a unique singularity.

\begin{example}
    Let $g\colon M_1\to M_1$ be a minimal diffeomorphism with positive topological entropy as in \cite{SYZ}. Let $h\colon M_2\to M_2$ be a diffeomorphism whose non-wandering set is formed by finitely many periodic orbits, i.e., $\mathcal{O}_h(p_1), \dots, \mathcal{O}_h(p_n)$. Notice that any Morse-Smale diffeomorphism satisfies the condition of $h$. Consider  $M=M_1\times M_2$ and $f=g\times h$. It is easy to see that $f$ is not minimal and $$\Omega(f)= \bigcup_{i=1}^{n} M_1\times \mathcal{O}_h(p_i),$$ where the family   $M_1\times \mathcal{O}_h(p_i)$  are minimal subsets, pairwise disjoint, and with positive topological entropy. Therefore, by Theorem \ref{extraresult} to obtain a singular suspension of $f$ with zero entropy, we need a brake  $\alpha$ similar to the brake obtained in \cite{SYZ}, but with at least $n$ zeros. \\

\end{example}

By the last example, we could wonder if $\Omega(f)$ is a disjoint union of minimal sets, then it could be enough to obtain a singular suspension $f$ with finitely or countable many singularities to vanish the topological entropy. The next example illustrates that it is also not enough.

\begin{example}
    Let $g$ be as in the previous example. Let $h\colon \mathbb{T}^2\to \mathbb{T}^2$ be defined by $h(x,y)=(x+y,y)$. Notice that fixes any horizontal fiber $S^1\times \{y\}$ is $h$-invariant. Moreover, $h|_{S^1\times\{y\}}$ is a rational circle rotation if $y$ is rational and irrational rotation otherwise. If we set $f=g\times h$, then $\Omega(f)$ is an uncountable union of pairwise disjoint minimal sets with positive entropy.  Therefore, by Theorem \ref{extraresult} we need uncountable many singularities to obtain a singular suspension with zero entropy.  
\end{example}

We end this work by discussing the strength of our techniques. First, we would like to remark that in Theorem \ref{extraresult} we do not mention the cardinality of the set of zeros in the function $\alpha$. We can then apply this result to obtain singular suspensions with uncountable many singularities and positive topological entropy. 

\begin{example}
   Let $f \colon M\to M$ be a transitive Anosov diffeomorphism. It is well known that every invariant and non-empty open subset of $M$ contains a horseshoe $\Lambda$ (see \cite{Katok}). Let $\Lambda$ be a horseshoe for  $f$. Let $\psi$ be a singular suspension of $f$ such that $\pi(Sing(\psi))=\Lambda$. This implies that $Sing(\psi)$ is uncountable.  Let $U=M\setminus\Lambda$. Then there is a horseshoe $\Lambda_2\subset  U$ and therefore $h(\psi)>0$, by Theorem \ref{extraresult}.     
\end{example}

    Using these ideas and mixing them with the proof of Theorem \ref{potencia} and Theorem \ref{extraresult},  we can prove that the only way of making the entropy of a singular suspension of an Anosov diffeomorphism vanishes is by making $\Omega\subset \pi(Sing(\psi))$.

\begin{proof}[\emph{\textbf{Proof of of Theorem \ref{Anosov-Susp}}}]
    Let $f$ be an Anosov diffeomorphism and $\psi$ a singular suspension of $f$. If $\pi(Sing(\psi))=M$ then $\gamma(x)=+\infty$, for every $x\in M$. Thus, Theorem \ref{equiv} implies $h(\psi)=0$.
    Conversely, suppose $h(\psi)=0$. Therefore, by Theorem \ref{extraresult} $A_{Sing}$ intersects any horseshoe of $f$. Therefore for every horseshoe $\Lambda$ of $f$, there is a point $x\in \Lambda$ and a singularity $\sigma_x$ such that $\pi(\sigma_x)=(x,0)$. Moreover, Theorem \ref{potencia} combined with \ref{extraresult} implies that these points $x$ are densely distributed in $\Lambda$. Since the horseshoes of an Anosov diffeomorphism are dense in $\Omega(f)$, we obtain a set $A\subset Sing(\psi)$ such that $\pi(A)$ is dense on $\Omega(f)$. Since $Sing(\psi)$ is closed, then $\Omega(f)= \overline{\pi(\overline{A})}\subset \pi(Sing(\psi))$.  
\end{proof}

Finally, notice that our main techniques are of a topological nature. Apart from the generic scenario, we can also obtain that the positiveness of topological entropy is preserved under singular suspension under the light of topological dynamics. The next result illustrate this point.

\begin{theorem}\label{shadcor}
    Let $f$ be a $C^1$-diffeomorphisms and let $K\subset M$ be a compact and $f$-invariant subset. Suppose $K$ is infinite, $f|_{K}$ is expansive and has the shadowing property. Let $\psi$ be a singular suspension with countable-many singularities, then both $f$ and $\psi$ have positive entropy. 
\end{theorem}

Since the proof of the previous result is just a reapplication of our techniques, we shall omit its proof to avoid unnecessary repetitions. To illustrate another advantage of considering only the topological dynamics of $f$ is that there is a concept of singular suspension of homeomorphisms (see \cite{Ko2} and \cite{Shi}). In the topological scenario, we lose the notion of a tangent vector field, but a suspension process can still be achieved. This process gives rise to a continuous singular flow with most of the properties of the smooth suspensions. Since the function, $\gamma$ in Subsection  \ref{singularsups} represents the first time in which a point on $(x,0)$ reaches $(f(x),0)$, then it can also be defined in the topological scenario. Thus, all the results in section \ref{characsection} also hold for continuous singular suspensions. In particular, Theorem  \ref{shadcor} also holds if $f$ is a homeomorphism.

\vspace{0.1in}


\noindent

{\em E. Rego}
\vspace{0.2cm}

Department of Mathematics

Southern University of Science and Technology

Shenzhen -Guangdong, China 

\email{elias@im.ufrj.br}

\vspace{0.1in}

{\em S. Roma\~{n}a}
\vspace{0.2cm}

School of Mathematics (Zhuhai)

Sun Yat-sen University

519802, PR China

\email{sergio@mail.sysu.edu.cn}

\noindent

\vspace{0.2cm}

\end{document}